\documentclass[11pt]{article}
\usepackage{graphicx,amsmath,amsthm}
\usepackage[font=small]{caption}

\usepackage{lineno}

\newtheorem{theorem}{Theorem}[section]
\newtheorem{lemma}[theorem]{Lemma}
\numberwithin{figure}{section}
\newcommand{\de}{\backslash}

\begin{document}

\title{On almost-planar graphs}

\author{Guoli Ding, Joshua Fallon, and Emily Marshall\\ 
Mathematics Department, Louisiana State University, Baton Rouge, LA 70803}

\date{\today}

\maketitle

\begin{abstract}
A nonplanar graph $G$ is called {\it almost-planar} if for every edge $e$ of $G$, at least one of $G\de e$ and $G/e$ is planar. In 1990, Gubser characterized 3-connected almost-planar graphs in his dissertation. However, his proof is so long that only a small portion of it was published. The main purpose of this paper is to provide a short proof of this result. We also discuss the structure of almost-planar graphs that are not 3-connected.
\end{abstract}

\section{Introduction}

A nonplanar graph $G$ is called {\it almost-planar} if for every edge $e$ of $G$, at least one of $G\de e$ and $G/e$ is planar. The following is an equivalent definition. For any specified set $\cal S$ of graphs, let us call a graph $\cal S$-{\it free} if it does not contain any graph in $\cal S$ as a minor. Using this terminology, we can say that a graph $G$ is almost-planar if and only if $G$ is not $\{K_5,K_{3,3}\}$-free but for every edge $e$ of $G$, at least one of $G\de e$ and $G/e$ is $\{K_5,K_{3,3}\}$-free.

The second definition is a special case of a concept in matroid theory. Let $\cal S$ be a specified set of matroids. Let {\it $\cal S$-free matroids} be defined analogous to $\cal S$-free graphs. Then a matroid $M$ is called $\cal S$-{\it fragile} if $M$ is not $\cal S$-free but for every element $e$ of $M$, at least one of $M\de e$ and $M/e$ is $\cal S$-free. In recent years, fragility has received a lot of attention due to its connection with representability. There are characterizations of $\cal S$-fragile matroids for several choices of $\cal S$. Among these, the first is a characterization of $\{K_5,K_{3,3}\}$-fragile 3-connected graphic matroids. This result was obtained by Gubser in his dissertation \cite{gubser2}, which was done before the term ``$\cal S$-fragile" was introduced. In the following discussion, we will use the term ``almost-planar graph" instead of ``$\{K_5,K_{3,3}\}$-fragile graph".

To state the characterization of Gubser we need a few definitions. Let $n\ge3$ be an integer. A {\it wheel} of size $n$, denoted by $W_n$, is the graph obtained from a cycle of length $n$ by adding a vertex and joining it to all vertices of the cycle. A {\it double wheel} of size $n$, denoted by $DW_n$, is the graph obtained from a cycle of length $n$ by adding two adjacent vertices and joining them to all vertices of the cycle. A {\it M\"obius ladder} of length $n$, denoted by $M_n$, is obtained from a cycle of length $2n$ by joining opposite pairs of vertices on the cycle. Finally, let $\cal W$ denote the set of all graphs constructed by identifying three triangles from three wheels. In other words, each graph $G\in\cal W$ admits a partition $(V_0,V_1,V_2,V_3)$ of its vertex set such that $G[V_0]$ is a triangle, $G[V_0\cup V_i]$ is a wheel $(i=1,2,3)$, and $G$ has no edges other than those in these three wheels.

\begin{theorem}[Gubser \cite{gubser2}] \label{thm:main}
Let $G$ be a 3-connected nonplanar graph. Then the following are equivalent. \\ 
\indent (i) $G$ is almost-planar; \\ 
\indent (ii) $G$ is a minor of a double wheel, a M\"obius ladder, or a graph in $\cal W$; \\ 
\indent (iii) $G$ is $\cal F$-free, where $\mathcal F = \{EX_1,EX_2,EX_3,EX_6,EX_8\}$ shown in Figure \ref{fig:exminors}.

\begin{figure}[ht]
\centerline{\includegraphics[scale=0.4]{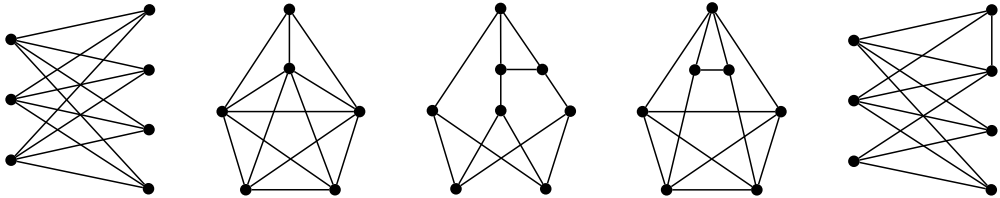}}
\caption{ Forbidden graphs $EX_1$, $EX_2$, $EX_3$, $EX_6$, and $EX_8$}
\label{fig:exminors}
\end{figure}
\end{theorem}

We remark that terminology $EX_1,EX_2,\dots,EX_8$ is inherited from \cite{gubser2}. We also remark that our formulation of this theorem is slightly different from that given in \cite{gubser2}. First, we modified the definition of graphs in $\cal W$. We make the change because our description better reveals the structure of these graphs. Second, set $\cal F$ given in \cite{gubser2} consists of eight graphs, instead of five graphs. These eight include the five listed in Figure \ref{fig:exminors} and another three graphs, which were denoted by $EX_4$, $EX_5$, and $EX_7$. Since each of these three extra graphs contains $EX_8$ as a minor, they are removed from our statement.  

By Seymour's Splitter Theorem \cite{splitter}, every almost-planar graph can be generated from $K_5$ and $K_{3,3}$ by repeatedly adding edges and splitting vertices, although not every graph constructed this way has to be almost-planar. Therefore, to prove that every almost-planar graph must be one of the three types listed in (ii), one only needs to show that any extension of any graph of the three types either results in another graph of the three types or contains a member of $\cal F$ as a minor. This is how the proof went in \cite{gubser2}. The author divided the analysis into fifteen cases, depending on which graph is extended and how the graph is extended. Since the case checking is lengthy, the published proof \cite{gubser1} includes only two of the fifteen cases. 

The main purpose of this paper is to provide a short proof of Theorem \ref{thm:main}. At the end of the paper, we will also discuss the structure of almost-planar graphs that are not 3-connected. This discussion corrects a few flaws appearing in \cite{gubser1}. 

We close this section by making two more remarks. Notice that graphs in $\cal W$ can be naturally divided into three groups, depending on how the three hubs are distributed on the common triangle. When none of the hubs are identified, the resulting graph is in fact a minor of a M\"obius ladder, as illustrated in Figure \ref{fig:mobius}. What this means is that we could define $\cal W$ differently so that these graphs are not included in $\cal W$. We choose to leave them in $\cal W$ since it makes the definition more clean. This is another difference between our formulation of Theorem \ref{thm:main} and the formulation given in \cite{gubser2}. 

\begin{figure}[ht]
\centerline{\includegraphics[scale=0.5]{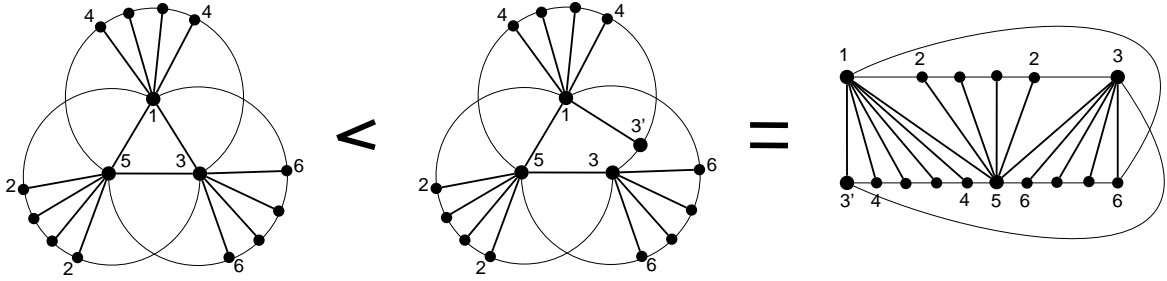}}
\caption{ Some graphs in $\cal W$ are minors of a M\"obius ladder}
\label{fig:mobius}
\end{figure}

Next, for each $G\in\{K_5,K_{3,3}\}$, let $G^+$ be obtained from $G$ by adding a new vertex $v$ and a new edge $e$ between $v$ and $G$. Then $G^+$ is not almost-planar since both $G^+\de e$ and $G^+/e$ are nonplanar. Notice that $EX_2$ contains $K_5^+$, while $EX_1$ and $EX_8$ contain $K_{3,3}^+$. Hence $EX_1$, $EX_2$, and $EX_8$ are not minimally non-almost-planar nonplanar graphs. However, by Theorem \ref{thm:main}, within the universe of 3-connected graphs, these three graphs are minimal and in fact, members of $\cal F$ are precisely the only five minimal graphs. In the last section of this paper, we determine all minimal graphs without imposing any connectivity. 

\section{A few lemmas}

In this section we present a few lemmas that will be used in our main proof. Let $k\ge0$ be an integer. A {\it $k$-separation} of a graph $G$ is an unordered pair $\{G_1,G_2\}$ of proper induced subgraphs of $G$ such that $G_1\cup G_2=G$ and $|G_1\cap G_2|=k$. The following is an immediate corollary of (2.4) of \cite{tripod}.

\begin{lemma}
Let $x$ be a cubic vertex of a nonplanar graph $G$. Suppose $G$ has no $k$-separation $\{G_1,G_2\}$ with $k\le2$ and such that some $G_i$ contains $x$ and all its three neighbors. Then $G$ has a subgraph $G'$ such that $G'$ is a subdivision of $K_{3,3}$ and $x$ is cubic in $G'$.
\label{lem:deg3}
\end{lemma}

A 3-connected graph $G$ with $|G|\ge5$ is called {\it internally $4$-connected} if for every 3-separation $\{G_1, G_2\}$ of $G$, exactly one of $G_1,G_2$ is $K_{1,3}$. Notice that no cubic vertex belongs to a triangle in an internally 4-connected graph. We will use this property repeatedly. Let $C_n^2$ be the graph obtained from a cycle of length $n$ by joining all pairs of vertices of distance two on the cycle.

\begin{lemma} 
Let $G$ be an internally 4-connected nonplanar minor of a M\"obius ladder. Then $G = M_n$ or $C_{2n-1}^2$ for some integer $n\ge3$.
\label{lem:minorM}
\end{lemma}

\begin{proof} 
Let $k\ge3$ be the smallest integer such that $G$ is a minor of $M_k$. Let $M_k$ consist of a Hamilton cycle $C$ and pairwise crossing chords. The minimality of $k$ implies that no chord is deleted in obtaining $G$. Moreover, since $G$ is nonplanar, no chord is contracted and no edge of $C$ is deleted. So $G$ is obtained by contracting some edges of $C$. Since no cubic vertex of $G$ belongs to a triangle, it follows that either $G=M_k$ or $k\ge5$ is odd and $G=C^2_k$.
\end{proof}

We remark that the lemma holds even if we drop the nonplanarity assumption. 

The next is a theorem of Maharry and Robertson \cite{maharry} on $M_4$-free graphs. 
An {\it alternating double wheel} of length $2n$ ($n\ge2$), denoted by $AW_{2n}$, is obtained from a cycle $v_1v_2...v_{2n}v_1$ by adding two new adjacent vertices $u_1,u_2$ such that $u_i$ is adjacent to $v_{2j+i}$ for all $i=1,2$ and $j=0,1,...,n-1$.

\begin{lemma}[Maharry and Robertson \cite{maharry}]
If an internally $4$-connected graph $G$ is $M_4$-free then at least one of the following holds.\\ 
\indent (i) $G$ is planar;\\ 
\indent (ii) $|G|\ge8$ and $G\de\{v_1,v_2,v_3,v_4\}$ is edgeless, for some distinct $v_1,v_2,v_3,v_4 \in V (G)$; \\ 
\indent (iii) $G=DW_{n+3}$ or $AW_{2n}$ for some $n\ge3$; \\ 
\indent (iv) $G$ is the line graph of $K_{3,3}$; \\ 
\indent (v) $G$ has fewer than eight vertices. 
\label{lem:M4free}
\end{lemma}

For any two vertices $x,y$ of a path $P$, let $P[x,y]$ denote the subpath of $P$ between $x,y$. 
Let $H$ be a graph. If $G$ is a subdivision of $H$, then {\it branch vertices} of $G$ are vertices of $G$ that correspond to vertices of $H$, and {\it arcs} of $G$ are paths of $G$ that correspond to edges of $H$. 
A {\it triad addition} of $H$ is obtained by adding a new vertex $v$ to $H$ and joining $v$ to three distinct vertices of $H$. 

\begin{lemma}\label{lem:t-add}
Let $H$ be a simple graph with $|H|\ge3$. If a $3$-connected graph $G$ has a subgraph $G'$ such that $G'$ is a subdivision of $H$ and $|G'|<|G|$, then $G$ contains a triad addition of $H$ as a minor.
\end{lemma}

\begin{proof} 
Let $v \in V(G)\setminus V(G')$ and let $P_1,P_2,P_3$ be three paths of $G$ from $v$ to $G'$ such that they are disjoint except for $v$. Let $v_i$ ($i=1,2,3$) be the other end of $P_i$. We first prove that $G',P_1,P_2,P_3$ can be chosen so that $v_1,v_2,v_3$ are not contained in a single arc of $G'$.

Suppose $G'$ has an arc $A$ containing $v_1,v_2,v_3$. Let $x,y$ be the two ends of $A$ and let $x,v_1,v_2,v_3,y$ be the order that they appear in $A$. Suppose $G',P_1,P_2,P_3$ are chosen such that $|A[x,v_1]\cup A[y,v_3]|$ is minimized. Since $G$ is 3-connected, $G\de\{v_1,v_3\}$ has a path $Q$ from $(P_1\cup P_2\cup P_3\cup A[v_1,v_3])\de\{v_1,v_3\}$ to $G'\de V(A[v_1,v_3])$. By the minimality of $|A[x,v_1]\cup A[y,v_3]|$, the other end of $Q$ must belong to $G'\de V(A)$. It follows that $G'\cup P_1\cup P_2\cup P_3\cup Q$ contains a subdivision of $H$ and three required paths.

Now we assume that $v_1,v_2,v_3$ are not contained in a single arc. It is easy to see that this property can be preserved when we contract $G'$ to $H$. At the end we obtain a triad addition of $H$ as a minor. 
\end{proof}

\section{A proof of Theorem~\ref{thm:main}}

Before proving Theorem \ref{thm:main} we establish three lemmas, which are the main parts of our proof. Let $\{G_1,G_2\}$ be a 3-separation of $G$. For $i=1,2$, let $G_i^\Delta$ be obtained from $G_i$ by adding all missing edges between vertices of $G_1\cap G_2$; let $G_i^Y$ be obtained from $G_i$ by adding a new vertex $v_i$ and joining $v_i$ to all three vertices of $G_1\cap G_2$. 

In the rest of this paper, we use $K_{3,3}^h$ to refer to the graph $EX_3$ and $K_5^h$ to refer to the graph $EX_6$ since these names better reflect the fact that $EX_3$ and $EX_6$ are formed from $K_{3,3}$ and $K_5$, respectively, by adding a handle edge.

\begin{lemma}
Let $G$ be $3$-connected and $\{K_{3,3}^+,K_{3,3}^h\}$-free. If $\{G_1,G_2\}$ is a $3$-separation of $G$ such that $G_1^Y$ is nonplanar, then $G_2^\Delta$ is a wheel.
\label{lem:wheel}
\end{lemma}

\begin{proof} Let $V(G_1) \cap V(G_2) = \{v_1,v_2,v_3\}$ and let $x$ be the cubic vertex added to $G_1$ in the formation of $G_1^Y$. Since $G$ is 3-connected and $G_1^Y$ is nonplanar, by Lemma~\ref{lem:deg3}, $G_1^Y$ contains a subdivision of $K_{3,3}$ such that $x$ is cubic in the subdivision. It follows that any cycle of $G_2$ disjoint from $\{v_1,v_2,v_3\}$ results in a $K_{3,3}^h$ minor in $G$. Thus $T = G_2\de \{v_1,v_2,v_3\}$ is a forest.

Let $u\in V(T)$. Since $G$ is 3-connected, $G_2$ contains three induced paths $P_{1},P_{2},P_{3}$ from $u$ to each of $v_1,v_2,v_3$, respectively, such that the three paths are disjoint except for $u$. If $G_2$ has a vertex outside $P_{1}\cup P_{2}\cup P_{3}$, then $G$ has a $K_{3,3}^{+}$ minor. So $T = (P_1\de v_1)\cup (P_2\de v_2)\cup (P_3\de v_3)$. If $P_1,P_2,P_3$ are all single edges, then $G_{2}^{\Delta}$ is a wheel $W_3$. Otherwise, suppose without loss of generality that $|P_1|\ge3$. 

Let $w$ be any internal vertex of $P_1$. Since $P_1$ is an induced path and $G$ is 3-connected, $w$ must be adjacent to some vertices outside $P_1$. Since $T$ is a tree, the only possible neighbors of $w$ outside $P_1$ are $v_2$ and $v_3$. If $w$ is adjacent to both $v_2,v_3$ then $G$ has a $K_{3,3}^+$ minor. So $w$ is adjacent to exactly one of $v_2,v_3$. If $w,w'$ are distinct internal vertices of $P_1$ such that both $wv_2$ and $w'v_3$ are edges of $G$, then $G$ again has a $K_{3,3}^+$ minor. Therefore, all internal vertices of $P_1$ are cubic and they have a common neighbor outside $P_1$, which we may assume to be $v_2$. 

Since $G$ is $K_{3,3}^+$-free, we deduce that $|P_2|=2$ and no internal vertex of $P_3$ is adjacent to $v_1$. So all vertices of $P_3\de v_3$ are cubic and adjacent to $v_2$. Therefore, $G_2$ consists of path $P_1\cup P_3$, all edges from $v_2$ to $(P_1\de v_1)\cup (P_3\de v_3)$, and possibly edges between $v_1,v_2,v_3$, which implies that $G_2^\Delta$ is a wheel.
\end{proof}

\begin{lemma}
Let $G$ be connected and $\{K_{3,3}^+,K_5^h\}$-free. If $G$ contains an $M_4$-minor then $G$ is a minor of $M_n$ for some $n\ge4$. 
\label{lem:hasM4}
\end{lemma}

\begin{proof}
Since $M_4$ is cubic, it is a topological minor of $G$, so $G$ has a subgraph $H$ isomorphic to a subdivision of $M_4$. Since $G$ is $K_{3,3}^+$-free, $H$ must be a spanning subgraph and no chord of $M_4$ is subdivided. It follows that the rim of $H$ is a Hamilton cycle $C$ of $G$. Let $v_0,...,v_7$ be the eight cubic vertices of $H$, listed in the order they appear on $C$. We denote the path of $C$ from $v_i$ to $v_{i+1}$ by $A_i$, where $i=0,...,7$. (In this proof the indices are taken modulo 8.)

Let $e$ be any chord of $C$. We prove that there exists $i$ such that the two ends of $e$ belong to $A_i$ and $A_{i+4}$, respectively. We first observe that no $A_i$ contains both ends of $e$ because then $H+e$ contains a $K_{3,3}^+$ minor. Next, assume one end of $e$ is an internal vertex of some $A_i$. If the other end of $e$ does not belong to $A_{i+4}$, it is straightforward to verify that $H+e$ can be contracted to $M_4+v_jv_{j+2}$, for some $j$, and thus $G$ contains a $K_{3,3}^+$ minor. So both ends of $e$ are cubic vertices of $H$, implying $e=v_iv_{i+t}$ ($t=3,4,5$), which satisfies the requirement. 

To finish the proof, we only need to show that any two non-adjacent chords $e,f$ must cross. By what we proved in the last paragraph, this is clear if the four ends of $e,f$ are not contained in $A_i\cup A_{i+4}$ for any $i$. If all ends of $e,f$ are contained in $A_i\cup A_{i+4}$ for some $i$ and if $e,f$ do not cross each other, then $H+e+f$ can be contracted to $M_4+v_iv_{i+5}+v_{i+1}v_{i+4}$, which contains a $K_5^h$ minor. Therefore, $e,f$ must cross and that completes our proof.
\end{proof}

\begin{lemma}
Let $G$ be internally $4$-connected, nonplanar, and $\{K_5^+,K_{3,3}^+,K_5^h,K_{3,3}^h\}$-free. Then $G$ is $M_n$ or $C_{2n+1}^2$ or $DW_n$ or $AW_{2n}$ for some $n\ge3$.
\label{lem:int4conn}
\end{lemma}

\begin{proof} If $G$ has an $M_4$-minor then the result follows immediately from lemmas \ref{lem:hasM4} and \ref{lem:minorM}. So we assume that $G$ is $M_4$-free. Since $G$ is nonplanar, one of (ii)-(v) of Lemma~\ref{lem:M4free} must hold. If (iii) holds then we are done. Case (iv) does not hold since $L(K_{3,3})$ contains a $K_{3,3}^h$ minor, as shown in Figure \ref{fig:lk33}.

\begin{figure}[ht]
\centerline{\includegraphics[scale=0.47]{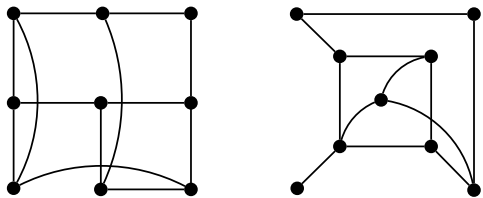}}
\caption{$K_{3,3}^h$ is a minor of $L(K_{3,3})$, and $K_{3,3}^+$ is a minor of Cube$+v$}
\label{fig:lk33}
\end{figure}

Suppose (ii) holds. Let $V(G)\de\{v_1,v_2,v_3,v_4\} = \{u_1,...,v_k\}$. Note that each $u_i$ has degree either 3 or 4. If two or more of them are of degree 4 then $K_{3,3}^+$ is a minor of $G$. If every $u_i$ is cubic, since no two of them have the same set of neighbors, it follows that $k=4$ and $G$ is the cube, which does not happen since $G$ is nonplanar. So we assume that $u_1$ has degree 4, while every other $u_i$ is cubic. As a result, $k=4$ or 5. If $k=4$ then $G$ is $AW_6$. If $k=5$ then $G\de u_1$ is the cube and so $G$ has a $K_{3,3}^+$ minor, as shown in Figure \ref{fig:lk33}.

It remains to consider case (v). If $|G| \leq 5$, since $G$ is nonplanar, we have $G \cong K_5 = DW_3$. If $|G|=6$, then $G$ contains a $K_{3,3}$ subgraph, as every 3-connected nonplanar graph with more than five vertices contains a $K_{3,3}$ minor. Since no cubic vertex of $G$ belongs to a triangle, it follows that either $G=K_{3,3}=M_3$, or $G$ contains at least two edges from each color class of $K_{3,3}$. If $G$ contains no other edges then $G=DW_4$; if $G$ contains at least one other edge then $G$ contains $K_5^+$.

If $|G|=7$, we may assume that $G$ contains a subgraph obtained from $K_{3,3}$ by subdividing an edge exactly once. Let $\{b_1,b_2,b_3\}$ and $\{r_1,r_2,r_3\}$ be the two color classes of $K_{3,3}$ and let $v$ be the vertex  subdividing edge $b_1r_1$. Since $v$ is not a cubic vertex that belongs to a triangle, we deduce that $v$ has degree at least 4. If $v$ is adjacent to both $b_2,b_3$ or both $r_2,r_3$, then $G$ contains $K_{3,3}^+$. So the degree of $v$ is 4, and we may assume that $v$ is adjacent to $b_2$ and $r_2$. Let $H$ be the subgraph of $G$ obtained from the subdivision of $K_{3,3}$ by adding edges $v b_2$ and $v r_2$. Now $b_1$ and $r_1$ are in triangles, so they must have degree at least 4 in $G$. If $G$ contains the edge $b_1 r_1$, then $G$ contains $K_{3,3}^+$. So $b_1$ must be adjacent to $b_2$ or $b_3$, and $r_1$ must be adjacent to $r_2$ or $r_3$. It follows that $b_3$ and $r_3$ are in triangles and thus they must also have degree at least 4 in $G$. Notice that $H+b_1b_3+r_1r_2 \cong H+ b_1b_3+r_2r_3$, which contains $K_5^h$. So $G$ is obtained from $H$ either by adding edges $b_1b_3$ and $r_1r_3$, which leads to $G=C_7^2$, or by adding paths $b_1b_2b_3$ and $r_1r_2r_3$, which leads to $G=DW_5$. 
\end{proof}

Now we are ready to prove Theorem \ref{thm:main}.

\begin{proof} [Proof of Theorem \ref{thm:main}] We prove implications \ (ii) $\Rightarrow$ (i) $\Rightarrow$ (iii)  $\Rightarrow$ (iv) $\Rightarrow$ (ii), \ where (iv) is the statement

(iv) {\it $G$ is $\{K_5^+,K_{3,3}^+,K_5^h,K_{3,3}^h\}$-free.} 

(ii) $\Rightarrow $ (i): Note that, if $\cal P$ is the class of all planar and almost-planar graphs, then $\cal P$ is closed under taking minors. Therefore, we only need to verify that each M\"obius ladder, each double wheel, and each graph in $\cal W$ is almost-planar. In $M_n,$ deleting a rim edge or contracting a chord results in a planar graph, so $M_n$ is almost-planar. In $DW_n$, deleting the axle or a rim edge or contracting a spoke results in a planar graph, so $DW_n$ is almost-planar. In any $G\in\cal W$, contracting a spoke or an edge in the common triangle or deleting a rim edge results in a planar graph, so $G$ is almost planar.

(i) $\Rightarrow $ (iii): Since $\cal P$ is closed under taking minors, we only need to show that each $H\in \cal F$ is not in $\cal P$. To do this, we only need to find an edge $e$ such that both $H\de e$ and $H/e$ are nonplanar.
In $EX_1$ every edge satisfies the requirement. In $EX_2$ and $EX_8$, any edge incident with the top cubic vertex (of the drawing shown in Figure~\ref{fig:exminors}) has the required property. In $K_{3,3}^h$ and $K_5^h$ ($EX_3$ and $EX_5$, respectively), the handle edge meets the requirement.  

(iii) $\Rightarrow$ (iv): Since $\{K_5^h,K_{3,3}^h\}\subseteq \cal F$, we only need to show that $G$ is $\{K_5^+, K_{3,3}^+\}$-free. Suppose otherwise. Then $G$ has a vertex $v$ such that $G\de v$ is nonplanar, which implies that $G$ has a subgraph $G'$ such that $|G'|<|G|$ and $G'$ is a subdivision of $K_5$ or $K_{3,3}$. By Lemma \ref{lem:t-add}, $G$ contains a triad addition of $K_5$ or $K_{3,3}$ as a minor. However, $EX_2$ is the only triad addition of $K_5$, and $EX_1,EX_8$ are the only triad additions of $K_{3,3}$. Thus (iii) is violated no matter what is the outcome, which proves (iv).

(iv) $\Rightarrow $ (ii): Assume $G$ is $\{K_5^+,K_{3,3}^+,K_5^h,K_{3,3}^h\}$-free. We will show $G$ is a minor of a double wheel, a M\"obius ladder, or a graph in $\cal W$. Recall that a {\it 3-sum} of two graphs $G_1,G_2$ is obtained by identifying a triangle of $G_1$ with a triangle of $G_2$, and then deleting 0, 1, 2, or 3 of the three identified edges.

If $G$ has three vertices whose deletion results in more than two components, since $G$ is $K_{3,3}^+$-free, there should be exactly three components and thus $G$ can be expressed as a 3-sum of three graphs $G_1,G_2,G_3$ over a common triangle. For each $i=1,2,3$, by applying Lemma \ref{lem:wheel} to the 3-separation of $G$ determined by $\{G_j\cup G_k, G_i\}$, where $\{i,j,k\}=\{1,2,3\}$, we conclude that $G_i$ is a wheel, which implies that $G\in\cal W$. In the rest of this proof we assume that any three vertices separate $G$ into at most two components, and we call such a graph {\it strongly connected}. 

We claim that $G$ has an internally 4-connected nonplanar minor $G'$ such that $G$ is obtained by 3-summing wheels to distinct triangles of $G'$. Suppose the claim is false. We consider a counterexample $G$ with $|G|$ as small as possible. Since internally 4-connected graphs are not counterexamples, $G$ must have a 3-separation $\{G_1,G_2\}$ such that neither part is $K_{1,3}$. Since $G$ is nonplanar, at least one of $G_1^Y$ and $G_2^Y$ is nonplanar. Without loss of generality, suppose $G_1^Y$ is nonplanar. Let us choose such a 3-separation with $G_1$ as small as possible. Since $G_2\ne K_{1,3}$, $G_1^\Delta$ is a 3-connected minor of $G$. Notice that $G_1^\Delta$ is also strongly connected, because any bad 3-cut of $G_1^\Delta$ would lead to a bad 3-cut of $G$. Moreover, $G_1^\Delta$ is nonplanar, because otherwise, since $G_1^Y$ is nonplanar, the common triangle $xyz$ of $G_1^\Delta$ and $G_2^\Delta$ would not be a facial triangle of $G_1^\Delta$, which implies that $G\de\{x,y,z\}$ has more than two components, contradicting the strong connectivity of $G$. Therefore, $G_1^\Delta$ satisfies all the assumptions of our claim and thus it is not a counterexample, as it is smaller than $G$. So $G_1^\Delta$ is obtained from an internally 4-connected nonplanar minor $G'$ by 3-summing wheels to different triangles of $G'$. By the minimality of $G_1$, triangle $xyz$ must be contained in $G'$ and thus $G$ is not a counterexample because $G$ is also obtained from $G'$ by 3-summing wheels to distinct triangles of $G'$, since $G_2^\Delta$ is a wheel by Lemma~\ref{lem:wheel}. This contradiction proves our claim. 


By Lemma~\ref{lem:int4conn}, we know $G'= M_n$ or $C_{2n+1}^2$ or $DW_n$ or $AW_{2n}$ for some $n\ge3$. Note that $C_{2n+1}^2$ is a minor of $M_{2n+1}$ and $AW_{2n}$ is a minor of $DW_{2n}$ so the theorem holds if $G=G'$. Next, suppose at least one wheel is added to $G'$. Since $M_n$ and $AW_{2n}$ do not have any triangles, we only need to consider all possible ways to add wheels to triangles of $C_{2n+1}^2$ and $DW_n$.  

Suppose $G'=C_{2n+1}^2$ ($n\ge3$). Observe that the 3-sum of $G'$ and a wheel is a graph with an $M_4$ minor. By Lemma \ref{lem:hasM4}, $G$ is a minor of a M\"obius ladder and thus the theorem holds.   

Suppose $G'=DW_n$ ($n\ge4$) is obtained by adding two adjacent vertices $u_1$ and $u_2$ to a cycle $C=v_0v_1...v_{n-1}v_n$ and joining each of $u_1,u_2$ to all $v_i$ (indices are taken modulo $n$ in this paragraph). Let $H$ be a 3-sum of $G'$ and a wheel $W_k$ over a triangle $T$. If $T$ is $u_1u_2v_i$ then $H$ contains a $K_{3,3}^+$ minor since $H\de \{v_{i+2}, v_{i+3}, ..., v_{i-2}\}$ contains a $K_{3,3}$ minor. There is only one other type of triangle in $G'$, so suppose $T$ is $u_1v_1v_2$. If $v_1v_2$ is an edge of $H$ then $H$ contains a $K_5^h$ minor, which can be seen by considering the graph formed by cycle $C$, cycle $v_0u_1v_3u_2v_0$, edges $u_1u_2,u_2v_1$, and three paths from $W_k\de V(T)$ to $T$. The same configuration also shows that if $k\ge 4$ then the hub of $W_k$ is not identified with $v_1$ or $v_2$. So the effect of 3-summing $W_k$ to $G'$ is to subdivide an edge $v_iv_{i+1}$ and to join all the subdividing vertices with some $u_j$. Many wheels can be added to $DW_n$ in the same fashion. Thus 3-summing wheels to $DW_n$ either creates a $K_{3,3}^+$ or $K_5^h$ minor or results in a minor of a double wheel. 

Finally, let $G'=K_5$ (note $K_5 \cong C_5^2 \cong DW_3$). We assume that $G$ is $M_4$-free because otherwise the theorem holds by Lemma \ref{lem:hasM4}. Let $\cal T$ be the set of {\it summing triangles}, which are triangles of $G'$ over which a wheel is 3-summed to $G'$ in obtaining $G$. To finish off this last case we first determine $\cal T$ and then determine how wheels are 3-summed to each member of $\cal T$. 

Let $V(G')=\{1,2,3,4,5\}$ and let $\mathcal T_1=\{123,124,134\}$ and $\mathcal T_2=\{123, 234, 145\}$. We claim that, up to a permutation, either $\mathcal T=\mathcal T_1$ or $\mathcal T\subseteq \mathcal T_2$. First we observe that turning triangles 123, 124, 125 into triads in $G'$ results in $K_{3,3}^h$, and turning 123, 134, 145 into triads results in $M_4$. Therefore, \{123, 124, 125\} and \{123, 134, 145\} are not subsets of $\cal T$ because otherwise $G$ contains $K_{3,3}^h$ or $M_4$ as a minor. Similarly \{123, 124, 134, 234$\}\not\subseteq \cal T$ because otherwise $G$ contains a $K_{3,3}^+$ minor (deleting any of the triads still leaves a nonplanar graph). Now it is straightforward to verify that either some vertex $i$ belongs to at least three summing triangles, in which case $\mathcal T=\mathcal T_1$ (up to a permutation), or every vertex $i$ belongs to at most two summing triangles, in which case $\mathcal T\subseteq \mathcal T_2$ (up to a permutation), and thus the claim is proved.

Suppose a wheel $W_k$ is 3-summed to $G'$ over $xyz\in\cal T$. If $k=3$ then at least one of the three edges of $xyz$ is not in $G$, because otherwise $G$ contains a $K_5^+$ minor. If $k>3$ and if the hub of $W_k$ is identified with $x$ then $yz\not\in E(G)$ because otherwise $G$ contains a $K_5^h$ minor (Figure \ref{fig:k5a}(i)). 

\begin{figure}[ht]
\centerline{\includegraphics[scale=0.44]{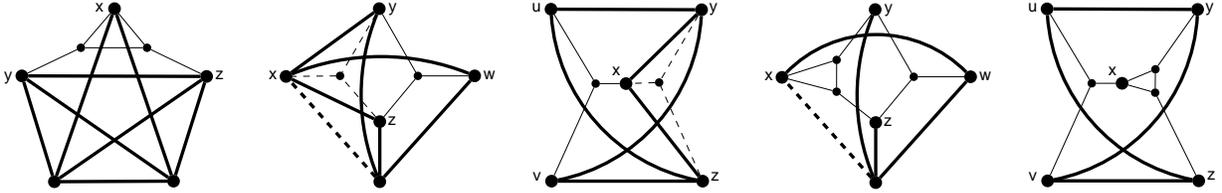}}
\caption{$G$ contains $K_5^h$, $K_{3,3}^+$, or $K_{3,3}^h$ }
\label{fig:k5a}
\end{figure}

We call $x$ a {\it hub} of $xyz$ if either $k>3$ and the hub of $W_k$ is identified with $x$, or $k=3$ and $x$ meets all edges of $G$ between $x,y,z$. The above observations imply that every summing triangle has a hub, and if $x$ is a hub then $yz\not\in E(G)$. In the following we make two observations. Let $xyz\in \cal T$. 

(a) {\it If $yzw\in \cal T$ then either $y$ or $z$ is a hub of $xyz$, where we assume $w\ne x$.} \\ 
\indent (b) {\it If $xuv \in \cal T$ then either $y$ or $z$ is a hub of $xyz$, where we assume $\{y,z\}\cap \{u,v\}=\emptyset$.}

\noindent To prove these two statements we assume that neither $y$ nor $z$ is a hub of $xyz$. Then either $k>3$ and the hub of $W_k$ is identified with $x$, or $k=3$ and both $xy,xz$ are edges of $G$. In both cases (a) and (b), if $k=3$ then $G$ contains a $K_{3,3}^+$ minor (Figure \ref{fig:k5a}(ii-iii)), and if $k>3$ then $G$ contains a $K_{3,3}^h$ minor (Figure \ref{fig:k5a}(iv-v)). These contradictions prove both (a) and (b).

Suppose $\mathcal T=\mathcal T_1$. Then none of the edges 23, 24, 34 are in $G$, because otherwise $G$ contains a $K_{3,3}^+$ minor. This observation and (a) imply that 1 is a hub for all three summing triangles. As a result, $G\de\{1,5\}$ is a cycle and thus $G$ is a minor of some $DW_n$.

Suppose $\mathcal T\subseteq \mathcal T_2$. We assume without loss of generality that $\mathcal T= \mathcal T_2$. We first note that 14 is not an edge of $G$ because otherwise $G$ contains a $K_{3,3}^h$ minor. This observation and (a) and (b) imply that 5 is a hub of 145, and hubs of 123, 234 are in $\{2,3\}$. If 2 (or 3) is a hub for both 123 and 234, then $G$ is a minor of some $DW_n$ (Figure \ref{fig:k5}(i)); if 2 is a hub of 123 and 3 is a hub of 234, then $G$ is a minor of a M\"obius ladder (Figure \ref{fig:k5}(ii)). Now the proof of Theorem \ref{thm:main} is complete. 
\begin{figure}[ht]
\centerline{\includegraphics[scale=0.5]{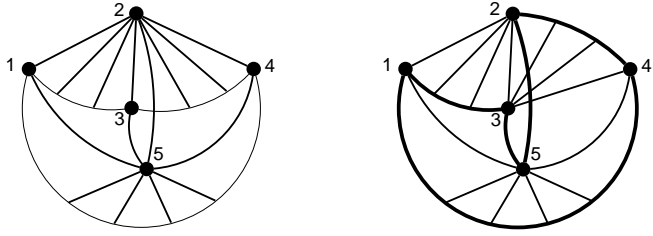}}
\caption{$G$ is a minor of $DW_n$ or a M\"obius ladder }
\label{fig:k5}
\end{figure}
\end{proof}

\section{Graphs of low connectivity}

In the published version \cite{gubser1} of \cite{gubser2}, graphs of low connectivity are also considered. Some of the statements in \cite{gubser1} are not accurate. Its second main theorem (Theorem 2.2) states: {\it A graph is neither planar nor almost-planar if and only if it has a $\{EX_i: 1\le i\le 8\}$-minor.} As we pointed out in the introduction, $K_{3,3}^+$ is a counterexample to this statement. In this section we prove a corrected version of this theorem. 

For any graph $G$, let $G\oplus e$ be obtained from $G$ by adding two adjacent new vertices, and let $G^*$ be obtained from $G$ by deleting all its isolated vertices (assuming that $E(G)\ne\emptyset$). Let $D(G)$ be the set of edges $e$ of $G$ such that $G\de e$ is planar. 

\begin{theorem}\label{thm:lowconn}
The following are equivalent for any nonplanar graph $G$. \\ 
\indent (i) $G$ is almost-planar;\\ 
\indent (ii) $G^*$ is obtained from a 3-connected almost-planar $H$ by subdividing edges in $D(H)$; \\ 
\indent (iii) $G$ is $\mathcal F'$-free, where $\mathcal F'=\{K_5^+, K_{3,3}^+, K_5^h, K_{3,3}^h,  K_5\!\oplus e, K_{3,3}\!\oplus e\}$. 
\end{theorem}

One possible sharpening of (ii) is to describe $D(H)$ explicitly. If $H$ is a M\"obius ladder or a double wheel or a graph in $\cal W$, then it is easy to determine $D(H)$. For each nonplanar minor $H'$ of $H$, one could also describe $D(H')$ because the structure of $H$ is simple enough. However, we choose not to include it in the paper since such a description is tedious and its derivation is straightforward. 
Statement (ii) is also touched in \cite{gubser1}, but the treatment is not rigorous. \\ 
\makebox[9mm]{} \parbox[t]{15cm}{\small  
Two corollaries of Theorem 2.1 characterize those almost-planar graphs that are not 3-connected. The elementary proofs are omitted. \\ 
{\bf Corollary 2.4.} {\it If $G$ is a connected, almost-planar graph, then $G$ is a series-parallel extension of a simple, 3-connected, almost-planar graph.} \\ 
{\bf Corollary 2.5.} {\it If $G$ is a disconnected, almost-planar graph, then $G$ is the union of a connected, almost-planar graph and a set of isolated vertices.}
} \medskip \\ 
In particular, both corollaries are not ``if and only if" type of statements, and (2.4) cannot be turned into such a statement. Our (ii) corrects both problems.

\begin{proof}[Proof of Theorem \ref{thm:lowconn}] 
We prove implications (ii) $\Rightarrow $ (i) $\Rightarrow $ (iii) $\Rightarrow $ (ii). 

(ii) $\Rightarrow $ (i): Since adding isolated vertices to an almost-planar graph results in an almost-planar graph, we only need to show that $G^*$ is almost-planar. This is clear because for each edge $e$ of $G^*$, $G^*\de e$ is planar if $e$ is an edge obtained by subdividing an edge of $D(H)$, while $G^*/e$ is planar if $e$ belongs to $E(H)\de D(H)$. 

(i) $\Rightarrow $ (iii): As in the proof of (i) $\Rightarrow $ (iii) of  Theorem~\ref{thm:main}, we only need to find, for each $H\in\cal F'$, an edge $e$ of $H$ such that both $H\de e$ and $H/e$ are nonplanar. The case $H=K_5^h$ or $K_{3,3}^h$ was settled in the proof of Theorem~\ref{thm:main}. If $H=K_5^+$, $K_{3,3}^+$, $K_5\!\oplus e$, or  $K_{3,3}\!\oplus e$, then the edge outside $K_5$ or $K_{3,3}$ satisfies the requirement.

(iii) $\Rightarrow$ (ii): Suppose the implication does not hold. We choose a counterexample $G$ with $|G|$ as small as possible. We first prove that $G$ is connected. If $G$ is disconnected, since $G$ is nonplanar, one component of $G$ contains a $K_{5}$ or $K_{3,3}$ minor. If another component of $G$ contains an edge, then $G$ contains $K_5\!\oplus e$ or $K_{3,3}\!\oplus e$  as a minor. So $G$ must have an isolated vertex $v$. By the minimality of $G$, $G\de v$ satisfies (ii), which implies $G$ satisfies (ii). This contradiction proves that $G$ is connected.

If $G$ has a 1-separation $\{G_1,G_2\}$, since $G$ is nonplanar, at least one $G_i$ contains $K_5$ or $K_{3,3}$ as a minor. Hence $G$ contains $K_{5}^{+}$ or $K_{3,3}^{+}$ as a minor. This contradiction proves that $G$ is 2-connected. 
 
Since 3-connected $\{K_5^+, K_{3,3}^+,K_5^h, K_{3,3}^h\}$-free nonplanar graphs are almost-planar, as shown in the proof of Theorem~\ref{thm:main} (implications (iv) $\Rightarrow$ (ii) $\Rightarrow$ (i)), $G$ cannot be 3-connected and thus $G$ has a 2-separation $\{G_1,G_2\}$. Let $V(G_1\cap G_2)=\{x,y\}$ and, for $i=1,2$, let $H_i=G_i$ or $G_i+xy$ in case $xy$ is not an edge of $G$. Since $G$ is nonplanar, we may assume without loss of generality that $H_1$ is nonplanar. Note that $G_1$ is planar because otherwise $G$ would contain a $K_5^+$ or $K_{3,3}^+$ minor. It follows that $xy\not\in E(G)$, since $H_1$ is nonplanar. 

Let $P$ be an induced path of $G_2$ between $x,y$. If $G_2$ has a vertex $v$ outside $P$, then $G\de v$ is nonplanar, because $G_1\cup P$, a subdivision of $H_1$, is nonplanar. This implies that $G$ contains a $K_5^+$ or $K_{3,3}^+$ minor, which is impossible. Therefore, $G_2$ must equal $P$ and thus $G$ is obtained from $H_1$ by subdividing edge $xy$. By the minimality of $G$, $H_1$ must satisfy (ii). Let $H_1$ be obtained from a 3-connected almost-planar graph $H$ by subdividing edges in $D(H)$. Each edge of $E(H)\de D(H)$ is not subdivided in the formation of $H_1$ and deleting such an edge in $H_1$ leaves a nonplanar graph. Therefore, $xy$ is not such an edge since $H_1\de xy=G_1$ is planar. It follows that there is an edge $e\in D(H)$ such that $xy$ belongs to a path (of $H_1$) obtained by subdividing $e$. Now it is clear that $G$ is obtained by repeatedly subdividing $e$, and thus $G$ satisfies (ii), which contradicts the assumption that $G$ is a counterexample. This contradiction completes our proof.
\end{proof}

\end{document}